\documentclass[10pt, reqno, a4paper]{amsart}
\usepackage{amssymb, amsmath, amsthm}
\usepackage{enumitem}
\usepackage[all]{xy}
\usepackage{mathtools}
\usepackage{extarrows}
\usepackage[normalem]{ulem}
\usepackage{stackrel}
\usepackage{mathrsfs}
\usepackage[english]{babel}
\usepackage{xcolor}

\usepackage{afterpage}

\usepackage{xcolor}
\usepackage[colorlinks=true, pdfstartview=FitV, linkcolor=blue, citecolor=blue, urlcolor=blue]{hyperref}

\newtheorem*{theorem*}{Theorem}

\newtheorem*{maintheorem*}{Main Theorem}

\newtheorem*{question*}{Question}

\theoremstyle{remark}

\theoremstyle{definition}

\setlist[1]{labelindent=\parindent, leftmargin=*}

\DeclareMathOperator{\car}{char}

\newcommand{\ZZ}{\mathbb{Z}}

\newcommand{\kk}{\mathrm{\mathbf{k}}}

\DeclareMathOperator{\KblSm}{K_{\mathrm{bl}}(\mathcal{S}m^{\mathrm{c}}_{\mathrm{\kk}})}
\DeclareMathOperator{\KV}{K_0(\mathcal{V}_{\kk})}
\DeclareMathOperator{\Ab}{\ZZ[\mathcal{A}_{\kk}]}

\title{
	Does the Grothendieck ring of varieties contain nilpotent elements?}
\author[A. Bot, A. Cangini, \and I. van Santen]
{Anna Bot, Alessio Cangini, \and Immanuel van Santen}

\date{\today}

\address{Klinik für Infektionskrankheiten und Spitalhygiene,
	Universitätsspital Zürich,\newline
	\indent Rämistrasse 100, CH-8091 Zürich, Switzerland}
\email{anna.bot@usz.ch}

\address{Departement Mathematik und Informatik, 
	Universit\"at Basel,\newline
	\indent Spiegelgasse 1, CH-4051 Basel, Switzerland}
\email{alessio.cangini@unibas.ch}

\address{Departement Mathematik und Statistik, 
	Universit\"at Bern,\newline
	\indent Sidlerstrasse 5, CH-3012 Bern, Switzerland}
\email{immanuel.van.santen@gmail.com}

\newcounter{claim}   
\renewcommand{\theclaim}{\arabic{claim}}

\begin{document}
\setcounter{tocdepth}{1}

\subjclass[2020]{14A10}
\keywords{Grothendieck ring of varieties, nilpotent elements}
\thanks{The three authors acknowledge support, respectively, from the Swiss 
National Science Foundation grants: ``Geometrically ruled surfaces'' (grant 
no. $200020\_192217$), ``Rational points, arithmetic dynamics, and heights'' 
(grant no. $200020\_219397$), and ``Symmetry in Large Algebraic Structures '' (grant no. $200020\_10000940$).}

\begin{abstract}
	We show that a ring closely related to the Grothendieck ring of varieties 
	has nilpotent elements, provided that the characteristic of the ground 
	field is equal to $11$ or at least $17$.
\end{abstract}

\maketitle

We work over an algebraically closed field $\kk$.
The \textit{Grothendieck ring $\KV$ of varieties over} $\kk$ is given by
all $\ZZ$-linear combinations of isomorphism classes of varieties defined over 
$\kk$ with the product $[X]\cdot[Y] \coloneqq [X\times_\kk Y]$ modulo the ideal 
generated by $[X] - [Z] - [X\setminus Z]$ where $Z$ is a closed subvariety of $X$.
The ring $\KblSm$ is given by all $\ZZ$-linear combinations of 
isomorphism classes of complete smooth varieties defined over $\kk$ 
with the product $[X]\cdot[Y] \coloneqq [X\times_\kk Y]$ modulo the ideal 
generated by $[\textrm{Bl}_Y X] - [E] - [X] + [Y]$ where $\pi \colon \textrm{Bl}_Y X \to X$ 
denotes the blow-up of $X$ 
in a closed smooth subvariety $Y \subseteq X$ and $E = \pi^{-1}(Y)$ denotes the 
exceptional divisor.

There is a natural homomorphism $\Psi \colon \KblSm \to \KV$.
If $\car{\kk} = 0$, by a remarkable result of Bittner~\cite[Theorem~3.1]{Bi2004The-universal-Eule},
$\Psi$ is an isomorphism.
We point out that in positive characteristic, $\Psi$ is an isomorphism as well 
if weak factorization of birational maps holds.

Assume that $\car\kk = 11$ or $\car\kk \ge 17$.
Joshi~\cite[Theorem~2.1.1]{Jo2025On-the-Grothendiec} showed that $\KblSm$ 
contains non-trivial zero divisors.
We remark that one can strengthen the argument to show:

\begin{theorem*}
	If $\car\kk = 11$ or $\car\kk\geq 17$, then
	the ring $\KblSm$ contains non-trivial nilpotent elements.
\end{theorem*}

\begin{proof}
	The assumption on the characteristic is equivalent to the existance of two 
	distinct supersingular elliptic curves~$E_1$ and~$E_2$~\cite[see Ch. V, 
	Theorem 4.1]{Si2009The-arithmetic-of-}.
	A noteworthy theorem of Deligne~\cite[see Theorem 3.5]{Sh1979Supersingular-K3-s} 
	implies that~$E_1 \times_{\kk} E_1$, $E_2 \times_{\kk} E_2$ and 
	$E_1 \times_{\kk} E_2$ are isomorphic.
	Now let~$\Ab$ denote the free abelian group generated by 
	isomorphism classes of abelian varieties over $\kk$ which becomes a ring via
	the product $[A]\cdot[A'] \coloneqq [A\times_{\kk} A']$. In $\Ab$, we have 
	\[
		([E_1]-[E_2])^2 = 
		[E_1 \times_{\kk} E_1] - 2 [E_1 \times_{\kk} E_2] + [E_2 \times_{\kk} E_2] = 0 \, .
	\]
	
	There is a natural homomorphism~$\KblSm \to \Ab$ which maps the class of a 
	smooth complete variety to the class of its Albanese~\cite[Theorem~2.2.1]{Sh1979Supersingular-K3-s}. 
	Note that this homomorphism has a natural section~$\Ab \to \KblSm$ and
	that by construction $[E_1]-[E_2] \neq 0$ inside $\Ab$. Hence, 
	the image of~$[E_1]-[E_2]$ in~$\KblSm$ is non-zero, and thus it 
	is our desired nilpotent element.
\end{proof}

This short note arose out of the attempt to construct nilpotent elements in 
$\KV$. So we finish with the following:

\begin{question*}
	Does the Grothendieck ring of varieties over $\kk$ contain nilpotent 
	elements?
\end{question*}

\par\bigskip
\renewcommand{\MR}[1]{}
\bibliographystyle{amsalpha}
\bibliography{refs}

\end{document}